\documentclass[12pt]{amsart}
\usepackage{times}
\usepackage{amsfonts}
\usepackage{amssymb}
\usepackage{bbm}
\usepackage{times}
\usepackage{amssymb}
\usepackage{amscd}
\usepackage{graphicx}

\usepackage{amsmath}
\usepackage{amssymb}

\usepackage{amsmath}
\usepackage{amsfonts}
\usepackage{amscd}
\usepackage{latexsym}
\usepackage{amsthm}
\usepackage{amssymb}
\usepackage{amsmath}
\usepackage{setspace}
\theoremstyle{plain}
\newtheorem{theorem}{Theorem}[section]
\newtheorem{corollary}[theorem]{Corollary}

\newtheorem{definition}[theorem]{Definition}

\def\ep{\varepsilon}

\begin{document}

\vskip 0.5cm

\title[non unital  general  tracially approximated ${\rm C^*}$-algebras] {non unital generalized tracially approximated ${\rm C^*}$-algebras}

\author{George A. Elliott, Qingzhai Fan, and Xiaochun Fang}

\address{George A. Elliott\\ Department of Mathematics\\ Univesity of Toronto\\ Toronto\\ Ontario \\ Canada  \hspace{0.1cm} M5S~ 2E4}
\email{elliott@math.toronto.edu}

\address{Qingzhai Fan\\ Department of Mathematics\\  Shanghai Maritime University\\
Shanghai\\China
\\  201306 }
\email{qzfan@shmtu.edu.cn}

\address{Xiaochun Fang\\ Department of Mathematics\\ Tongji
University\\
Shanghai\\China
\\200092}
\email{xfang@tongji.edu.cn}

\thanks{{\bf Key words}  ${\rm C^*}$-algebras, tracial approximation, tracially $\mathcal{Z}$-absorbing.}

\thanks{2000 \emph{Mathematics Subject Classification\rm{:}} 46L35, 46L05, 46L80}

\begin{abstract}
Let $\Omega$ be a class of ${\rm C^*}$-algebras. In this paper, we  study  a class of not necessarily unital  generalized  tracial approximation  ${\rm C^*}$-algebras, and  the class of simple  ${\rm C^*}$-algebras which can be generally tracially
approximated by ${\rm C^*}$-algebras in $\Omega$, denoted by  ${\rm gTA}\Omega$. Let  $\Omega$ be a class of  unital ${\rm C^*}$-algebras and let $A$ be a simple unital ${\rm C^*}$-algebra. Then  $A\in {\rm gTA}\Omega$, if, and only if, $A\in {\rm WTA}\Omega$ (where ${\rm WTA}\Omega$ is the class of   weakly  tracially  approximable unital ${\rm C^*}$-algebras introduced by Elliott, Fan, and Fang).
Consider the class of  ${\rm C^*}$-algebras which  are tracially $\mathcal{Z}$-absorbing (or are of  tracial nuclear dimension at most $n$, or are  $m$-almost divisible, or have the property $\rm SP$). Then $A$  is  tracially $\mathcal{Z}$-absorbing  (respectively, has   tracial nuclear dimension at most $n$, is weakly ($n, m$)-almost divisible, has the property $\rm SP$) for  any  simple  ${\rm C^*}$-algebra $A$ in the corresponding  class of generalized tracial approximation ${\rm C^*}$-algebras.
\end{abstract}
\maketitle

\section{Introduction}
 The Elliott program for the classification of amenable
 ${\rm C^*}$-algebras might be said to have begun with the ${\rm K}$-theoretical
 classification of AF algebras in \cite{E1}. A major next step was the classification of simple
  AH algebras without dimension growth (in the real rank zero case see \cite{E6}, and in the general case
 see \cite{GL}).
  This led eventually to the classification of  simple separable amenable ${\rm C^*}$-algebras with finite nuclear dimension in the UCT class (see \cite{KP}, \cite{PP2}, \cite{EZ5}, \cite{GLN1}, \cite{GLN2}, \cite{TWW1}, \cite{EGLN1},  \cite {GL2}, \cite{GL4}, and \cite{GL3}).

A crucial intermediate step was Lin's axiomatization of  Elliott-Gong's decomposition theorem for  simple AH algebras of real rank zero (classified by Elliott-Gong in \cite{E6}) and Gong's decomposition theorem (\cite{G1}) for simple AH algebras (classified by Elliott-Gong-Li in \cite{GL}). For this purpose, Lin introduced the concepts of  TAF and TAI (\cite{L0} and \cite{L1}). Instead of assuming inductive limit structure, Lin started with a certain abstract (tracial) approximation property. Elliott and  Niu in \cite{EZ} considered  this  notion of
 tracial approximation by  other classes of unital ${\rm C^*}$-algebras than the finite-dimensional ones  for TAF and the interval algebras for TAI.

Large and centrally  large subalgebras were introduced in \cite{P3} and \cite{AN} by Phillips and  Archey   as  abstractions of Putnam's orbit breaking subalgebra of the crossed product algebra ${\rm C}^*(X,\mathbb{Z},\sigma)$ of the Cantor set by a minimal homeomorphism in \cite{P}.

Inspired  by  centrally large subalgebras and tracial approximation ${\rm C^*}$-algebras, Elliott, Fan, and Fang formulated the notation  unital weakly   tracial approximation ${\rm C^*}$-algebra in \cite{EFF}, generalizing  both Archey and Phillips's centrally large subalgebras and Lin's notion of tracial approximation.

Let $\Omega$ be a class of unital ${\rm C^*}$-algebras. Elliott, Fan, and Fang in \cite{EFF} introduced   the class
 of unital ${\rm C^*}$-algebras which can be weakly  tracially approximated by ${\rm C^*}$-algebras in $\Omega$, and denoted this class by ${\rm WTA}\Omega$.

\begin{definition}\label{def:1.1}(\cite{EFF})  A  simple unital ${\rm C^*}$-algebra $A$  is  said to belong to the class ${\rm WTA}\Omega$  if, for any
 $\varepsilon>0$, any finite
subset $F\subseteq A$, and any  non-zero element $a\geq 0$, there
exist a  projection $p\in A$, an  element $g\in A$ with $0\leq g\leq 1$,
  and a unital ${\rm C^*}$-subalgebra $B$ of $A$ with
$g\in B, 1_B=p$, and $B\in \Omega$, such that

$(1)$  $(p-g)x\in_{\varepsilon} B, ~ x(p-g)\in_{\varepsilon} B$ for all $x\in  F$,

$(2)$ $\|(p-g)x-x(p-g)\|<\varepsilon$ for all $x\in F$,

$(3)$ $1-(p-g)\precsim a$ (see Section 2), and

$(4)$ $\|(p-g)a(p-g)\|\geq \|a\|-\varepsilon$.
\end{definition}

Inspired by   Lin, Elliott, and Niu's  tracial approximation structure, by Castillejos, Li,  and Szabo's    tracial $\mathcal{Z}$-stability of simple non-unital ${\rm C^*}$-algebras in \cite{CKS}, by  Amini, Golestani, Jamali, and Phillips's  simple tracially $\mathcal{Z}$-absorbing ${\rm C^*}$-algebras in \cite{AGJC}, and by Elliott, Fan, and Fang's  weak tracial approximation of unital  ${\rm C^*}$-algebras in \cite{EFF}, in this paper, we shall introduce a class of not necessarily unital   generalized tracial approximation ${\rm C^*}$-algebras.

 Let  $\Omega$ be a class of  ${\rm C^*}$-algebras. We define as follows  the class
 of ${\rm C^*}$-algebras which can be   tracially approximated  in the generalized sense by ${\rm C^*}$-algebras in $\Omega$, and denote this class by ${\rm gTA}\Omega$.

\begin{definition}\label{def:1.2}
Let $A$ be a simple ${\rm C^*}$-algebra. $A$ will be said to
      belong to the class ${\rm gTA}\Omega$ if,  for any $\varepsilon>0$, every finite set $F\subseteq A$, and for every pair of  positive elements $a, y\in A$ with $a\neq 0$,  there exist a ${\rm C^*}$-subalgebra $B$ of $A$ with $B\in \Omega$,  and a positive element $g\in B$ with $\|g\|\leq 1$,  such that the following conditions hold:

$(1)$ $\|xg-gx\|<\ep$ for all $x\in  F$,

$(2)$ $gx\in_\ep B, xg\in_\ep B$ for
all $x\in  F$,

 $(3)$ $(y^{2}-ygy-\varepsilon)_{+}\precsim a$, and

$(4)$ $\|gag\|\geq \|a\|-\varepsilon$.
  \end{definition}

 Let  $\Omega$ be a class of  unital ${\rm C^*}$-algebras and $A$ be a simple unital ${\rm C^*}$-algebra. As  we shall show,  $A\in {\rm gTA}\Omega$ if, and only if, $A\in {\rm WTA}\Omega$. (Theorem \ref{thm:2.14})

We shall prove the following four main results:

 Let $\Omega$ be a class of
${\rm C^*}$-algebras which   are tracially $\mathcal{Z}$-absorbing
(a property introduced by  Amini, Golestani, Jamali, and Phillips in  \cite{AGJC} and Castillejos, Li,  and Szabo in \cite{CKS}; see Definition \ref{def:2.6}).   Then  $A$  is  tracially $\mathcal{Z}$-absorbing  for  any   simple  ${\rm C^*}$-algebra $A\in{\rm gTA}\Omega$. (Theorem \ref{thm:3.1})
\vskip 0.1cm

 Let $\Omega$ be a class of
${\rm C^*}$-algebras which have  the second type of tracial nuclear dimension at
    	most $n$ (a property introduced by  Fu in \cite{FU}; see Definition \ref{def:2.8}).  Then  $A$ has  the second type of  tracial nuclear dimension at	most $n$  for  any   simple  ${\rm C^*}$-algebra $A\in{\rm gTA}\Omega.$ (Theorem \ref{thm:3.3}.)
\vskip 0.1cm

Let $\Omega$ be a class of
${\rm C^*}$-algebras  which are  $m$-almost divisible (a property introduced by Robert and Tikuisis in \cite{RT}; see Definition \ref{def:2.3}).   Let  $A\in{\rm gTA}\Omega$ be a  simple   stably finite ${\rm C^*}$-algebra such that for any $n\in \mathbb{N}$ the
 ${\rm C^*}$-algebra  $\rm{M}$$_n(A)$  belongs to the class ${\rm gTA}\Omega$. Then $A$  is weakly $(2, m)$-almost divisible (see Definition \ref{def:2.4}). (Theorem \ref{thm:3.4}.)
\vskip 0.1cm

Let $\Omega$ be a class of  ${\rm C^*}$-algebras with the property $\rm SP$.  Then  $A$  has the  property $\rm SP$  for  any  simple ${\rm C^*}$-algebra $A\in{\rm  gTA}\Omega.$ (Theorem \ref{thm:3.5}.)

 \section{Preliminaries and definitions}
Recall that  a $\rm C^*$-algebra  $A$ has the  property $\rm SP$ if every non-zero
hereditary $\rm C^*$-subalgebra of $A$ contains a non-zero projection.

Let $A$ be a ${\rm C^*}$-algebra, and let ${\rm M}_n(A)$ denote the ${\rm C^*}$-algebra of  $n\times n$ matrices with entries
elements of $A$. Let ${\rm M}_{\infty}(A)$ denote the algebraic inductive  limit of the sequence  $({\rm M}_n(A),\phi_n),$
where $\phi_n:{\rm M}_n(A)\to {\rm M}_{n+1}(A)$ is the canonical embedding as the upper left-hand corner block.
 Let ${\rm M}_{\infty}(A)_+$ (respectively, ${\rm M}_{n}(A)_+$) denote
the positive elements of ${\rm M}_{\infty}(A)$ (respectively, ${\rm M}_{n}(A)$).  Given $a, b\in {\rm M}_{\infty}(A)_+,$
one says  that $a$ is Cuntz subequivalent to $b$ (written $a\precsim b$) if there is a sequence $(v_n)_{n=1}^\infty$
of elements of ${\rm M}_{\infty}(A)$ such that $$\lim_{n\to \infty}\|v_nbv_n^*-a\|=0.$$
One says that $a$ and $b$ are Cuntz equivalent (written $a\sim b$) if $a\precsim b$ and $b\precsim a$. We  shall write $\langle a\rangle$ for the Cuntz equivalence class of $a$.

Let $A$ be a  unital ${\rm C^*}$-algebra. Recall \cite{APT} that a positive element $a\in A$ is called purely positive if $a$ is
not Cuntz equivalent to a projection. Let $A$ be a stably finite  ${\rm C^*}$-algebra (c.f., \cite{L2}) and let  $a\in A$ be  a positive element. Then either $a$  is a purely positive element or $a$ is Cuntz equivalent to a projection (c.f., \cite{APT}).

Given $a$ in ${\rm M}_{\infty}(A)_+$ and $\varepsilon>0,$ we shall denote by $(a-\varepsilon)_+$ the element of ${\rm C^*}(a)$ corresponding (via the functional calculus) to the function $f(t)={\max (0, t-\varepsilon)},~~ t\in \sigma(a)$. By the functional calculus, it follows in a straightforward manner that $((a-\varepsilon_1)_+-\varepsilon_2)_+=(a-(\varepsilon_1+\varepsilon_2))_+.$

The following facts are  well known.
\begin{theorem}{\rm (\cite{PPT}, \cite{HO}, \cite{P3}, \cite{RW}.)} \label{thm:2.1} Let $A$ be a ${\rm C^*}$-algebra.

 $(1)$ Let $a, b\in A_+$. Let  $\varepsilon>0$  be such that
$\|a-b\|<\varepsilon$.  Then there exists a contraction $d$ in $A$ with $(a-\varepsilon)_+=dbd^*$.

$(2)$ Let $a, p$ be positive elements in ${\rm M}_{\infty}(A)$ with $p$ a projection. If $p\precsim a,$ then there is $b$ in ${\rm M}_{\infty}(A)_+$ such that  $bp=0$ and $b+p\sim a$.

 $(3)$ Let $a$ be a  positive element  of $A$
not Cuntz equivalent to a projection. Let  $\delta>0$,  and let $f\in C_0(0,1]$ be a non-negative function with $f=0$ on $(\delta,1),$  $f>0$ on $(0,\delta)$,
and $\|f\|=1$.  Then $f(a)\neq 0$
and  $(a-\delta)_++f(a)\precsim a.$

$(4)$ Let $a, b\in A$ satisfy $0\leq a\leq b$. Let $\varepsilon\geq 0$.
Then $(a-\varepsilon)_+\precsim(b-\varepsilon)_+$ (Lemma 1.7 of \cite{P3}).

$(5)$ Let $c\in A$ and $\varepsilon>0$.
Then $(cc^*-\varepsilon)_+\sim (c^*c-\varepsilon)_+$.
\end{theorem}

The following  theorem is  Lemma 2.3 of \cite{AGJC}.

\begin{theorem}{\rm (\cite{AGJC}.)} \label{thm:2.2} Let $A$ be a ${\rm C^*}$-algebra, let $x\in A$ be nonzero, and let $b\in A_+$. Then for any $\varepsilon>0$,
$$(xbx^*-\varepsilon)_+\precsim x(b-\varepsilon/\|x\|^2)_+\precsim (b-\varepsilon/\|x\|^2)_+.$$
If $\|x\|\leq 1$ then
$$(xbx^*-\varepsilon)_+\precsim x(b-\varepsilon)_+x^*\precsim (b-\varepsilon)_+.$$
\end{theorem}

The property  of   $m$-almost divisibility  was  introduced by Robert and Tikuisis in \cite{RT}.

\begin{definition}(\cite{RT}.)\label{def:2.3} Let $A$ be a $\rm C^*$-algebra.  Let $m\in \mathbb{N}$. We say that $A$ is  $m$-almost divisible if, for each $a\in {\rm M}_{\infty}(A)_+,$  for any integer $k\in \mathbb{N}$,  and any $\varepsilon>0,$ there exists $b\in{\rm M}_{\infty}(A)_+$ such that $k\langle b\rangle\leq \langle a\rangle$ and $\langle (a-\varepsilon)_+\rangle\leq (k+1)(m+1)\langle b\rangle$.
\end{definition}

The property  of   weak   ($n, m$)-almost divisibility  was  introduced by Elliott, Fan, and Fang in \cite{EFF}.
\begin{definition}(\cite{EFF}.)\label{def:2.4} Let $A$ be a $\rm C^*$-algebra. Let $n,~ m\in \mathbb{N}$. We say that $A$ is weakly   ($n, m$)-almost divisible if for each $a\in {\rm M}_{\infty}(A)_+,$  for any integer $k\in \mathbb{N}$,  and any $\varepsilon>0,$ there exists $ b\in{\rm M}_{\infty}(A)_+$ such that $k\langle b\rangle\leq \langle a\rangle+\langle a\rangle$ and $\langle (a-\varepsilon)_+\rangle\leq (k+1)(m+1)\langle b\rangle$.
\end{definition}

Note that if $A$ has the property of either Definition \ref {def:2.3} or
Definition \ref {def:2.4} then so also does any matrix algebra over $A$.

Hirshberg and  Orovitz  introduced  the notion of  tracial $\mathcal{Z}$-absorption  in \cite{HO}.
Also in \cite{HO}, Hirshberg and  Orovitz showed that tracial $\mathcal{Z}$-absorption is equivalent to $\mathcal{Z}$-absorption for any simple unital separable nuclear $\rm C^*$-algebra.

\begin{definition}\rm {(\cite{HO}.)}\label{def:2.5} We say a unital ${\rm C^*}$-algebra $A$ is tracially $\mathcal{Z}$-absorbing, if $A\neq {{\mathbb{C}}}$
and for any finite set $F\subseteq A,$  any $\varepsilon>0,$  any  non-zero positive element $a\in A$,
and any integer $n\in {{\mathbb{N}}},$ there exists  an order zero  completely positive contraction   map $\psi: {\rm M}_n\to A$, where  order zero means preserving  orthogonality, i.e.,
 $\psi(e)\psi(f)=0$ for all $e,f\in {\rm M}_n$ with $ef=0$, such that
the following properties hold:

$(1)$ $1-\psi(1)\precsim a,$ and

$(2)$ for any normalized element $x\in {\rm M}_n$ (i.e., with $\|x\|=1$) and any $y\in F$ we have $\|\psi(x)y-y\psi(x)\|<\varepsilon$.
\end{definition}

Inspired by Hirshberg and  Orovitz's tracial $\mathcal{Z}$-absorption for  simple unital
 ${\rm C^*}$-algebra,   Castillejos,  Li, and Szabo in \cite {CKS},  and property of Amini,  Golestani, Jamali, and Phillips in \cite{AGJC}, introduced the tracial $\mathcal{Z}$-absorption for a simple non-unital   ${\rm C^*}$-algebra.

\begin{definition}{\rm ({\cite{AGJC}, \cite{CKS}}.)}\label{def:2.6}  A simple  ${\rm C^*}$-algebra $A$ is tracially $\mathcal{Z}$-absorbing, if $A\neq {{\mathbb{C}}}$
and for any finite set $F\subseteq A,$ any  $\varepsilon>0,$   any non-zero pair of  positive element $a, b\in A$,
and any integer $n\in {{\mathbb{N}}},$ there is an order zero  completely positive contraction    map $\psi: {\rm M}_n\to A$ such that
the following properties hold:

$(1)$ $(b^2-b\psi(1)b-\varepsilon)_+\precsim a,$ and

$(2)$ for any normalized element $x\in {\rm M}_n$ (i.e., with $\|x\|=1$) and any $y\in F$ we have $\|\psi(x)y-y\psi(x)\|<\varepsilon$.
\end{definition}

Remark: By \cite{AGJC} or  \cite{CKS}, if $A$ is a  simple unital ${\rm C^*}$-algebra, then   Definition \ref{def:2.5} is equivalent to Definition \ref{def:2.6}.

Winter and Zacharias  introduced the notion of  nuclear dimension for  ${\rm C^*}$-algebras in \cite{WW3}.

\begin{definition}{\rm (\cite{WW3}.)}\label{def:2.7} Let $A$ be a ${\rm C^*}$-algebra and let $m\in {\mathbb{N}}$ be an integer.
 A completely  positive contraction  map $\varphi:F\to A$ is  $m$-decomposable (where $F$ is a
finite-dimensional  ${\rm C^*}$-algebra), if there is  a decomposition $F=F^{(0)}\oplus F^{(1)}
\oplus\cdots \oplus F^{(m)}$ such that the restriction $\varphi^{(i)}$ of $\varphi$ to $F^{(i)}$
has order zero (which means preserves orthogonality, i.e.,
 $\psi(e)\psi(f)=0$ for all $e, f\in {\rm M}_n$ with $ef=0$), for each $i\in \{0,\cdots,$$ m\}$.  $\varphi$ is said to be
$m$-decomposable with respect to  the decomposition $F=F^{(0)}\oplus F^{(1)}
\oplus\cdots \oplus F^{(m)}$.
 A has nuclear dimension $m$,  written ${\rm dim_{nuc}}(A)=m$, if $m$ is the least integer such that
the following condition holds: For any finite subset $G\subseteq A$ and $\varepsilon>0$, there is
a finite-dimensional completely  positive   approximation $(F,\varphi, \psi)$
for $G$ to within $\varepsilon$ (i.e., $F$ is finite-dimensional, $\psi: A\to F$  and
$\varphi:F\to A$ are completely positive,  and $\|\varphi\psi(b)-b\|<\varepsilon$
for any $b\in G$) such that $\psi$ is a   contraction, and $\varphi$ is $m$-decomposable with order zero
 completely positive contraction  components $\varphi^{(i)}$. If no such $m$ exists, we write
${\rm dim_{nuc}}(A)=\infty$.
 \end{definition}

Inspired by  Hirshberg and  Orovitz's  tracial $\mathcal{Z}$-absorption in \cite{HO}, Fu introduced
 a notion of  tracial nuclear dimension in his doctoral dissertation  \cite{FU} (see also  \cite{FL}).

\begin{definition}{\rm (\cite{FL}.)}\label{def:2.20}
		Let $A$ be a unital simple $\rm C^{*}$-algebra.  Let $n\in \mathbb{N}$ be an integer. $A$ is said to have tracial nuclear dimension at most $n$,
		denoted by ${\rm Tdim_{nuc}} A\leq n$, if for any finite subset $\mathcal{F}\subseteq A$, for any $\varepsilon>0$, and for any nonzero positive element $a\in A$, there exist a finite-dimensional $\rm C^{*}$-algebra $F$, a  completely positive contractive  map $\alpha:A\rightarrow F$,
		a nonzero piecewise contractive $n$-decomposable completely positive  map $\beta:F\rightarrow A$, and a completely positive contractive map
		$\gamma:A\rightarrow A\cap\beta^{\perp}(F)$ such that
		
		 $(1)$ $\|x-\gamma(x)-\beta\alpha(x)\|<\varepsilon$ for all $x\in\mathcal{F}$, and
		
		 $(2)$ $\gamma(1_{A})\precsim a$.
	\end{definition}
Note that tracial nuclear dimension at most $n$ is preserved  under tensoring with matrix algebras and under passing to unital hereditary $\rm C^{*}$-subalgebras.

\begin{definition}{\rm (\cite{FU}.)}\label{def:2.8}
	Let $A$ be a $\rm C^{*}$-algebra. Let $n\in \mathbb{N}$ be an integer. $A$ is said to have the second type of  tracial nuclear dimension
	at most $n$, and denoted by ${\rm T^{2}dim_{nuc}} A\leq n$, if for any finite positive subset $\mathcal{F}\subseteq A$, for any $\varepsilon>0$, and for any nonzero positive element $a\in A$, there exist a finite-dimensional $\rm C^{*}$-algebra $F=F_{0}\oplus\cdots\oplus F_{n}$ and completely positive  maps $\psi:A\rightarrow F$, $\varphi:F\rightarrow A$ such that
	
	$(1)$  for any $ x\in F$, there exists $x'\in A_{+}$, such that $x'\precsim_{A}a$,  and $\|x-x'-\varphi\psi(x)\|<\varepsilon$,
	
	$(2)$ $\|\psi\|\leq1$, and
	
	$(3)$ $\varphi|_{F_{i}}$ is a contractive completely positive  order zero map for $i=1, \cdots, n$.
	
\end{definition}

Note that the second type of  tracial nuclear dimension  at most $n$ is preserved  under tensoring with matrix algebras and under passing to unital hereditary $\rm C^{*}$-subalgebras.

Let $A$ be a unital $\rm C^{*}$-algebra. It is easy to see that
 ${\rm {Tdim_{nuc}}}(A)$ $\leq n$ implies   ${\rm {T^2dim_{nuc}}}(A)\leq n$.

 \begin{theorem}\label{thm:2.12}
 Let $\Omega$ be a class of ${\rm C^*}$-algebras. Let $A$ be a simple unital ${\rm C^*}$-algebra. Then $A\in {\rm gTA}\Omega$ (see Definition \ref{def:1.2}) if, and only if, for  every $\varepsilon>0$, every finite set $F\subseteq A$, and  every nonzero  positive element $a\in A$, there exist a ${\rm C^*}$-subalgebra $B$ of $A$ with $B\in \Omega$,  and a positive element $g\in B$ with $\|g\|\leq 1$, such that the following conditions are satisfied:

$(1)$ $\|xg-gx\|<\ep$, for all $x\in  F$,

$(2)$ $gx\in_\ep B, xg\in_\ep B$, for
all $x\in F$,

 $(3)$ $(1-g-\varepsilon)_+\precsim  a$, and

$(4)$ $\|gag\|\geq \|a\|-\varepsilon$.

 \end{theorem}
 \begin{proof}
 $\Longrightarrow$:
 For  any  $\varepsilon>0$,
any finite set $F\subseteq A$, and  any pair of  positive elements $a,  1_{A}\in A$ with $a\neq 0$, since $A\in {\rm gTA}\Omega$, there exist a
${\rm C^*}$-subalgebra $B$ of $A$  with $B\in \Omega$,   and a  positive element  $g\in B$
with $\|g\|\leq 1$,
such that

$(1)$ $\|xg-gx\|<\ep$ for all $x\in  F$,

$(2)$ $gx\in_\ep B, xg\in_\ep B$ for
all $x\in  F$,

 $(3)$ $(1_{A}-g-\varepsilon)_+\precsim a$, and

$(4)$ $\|gag\|\geq \|a\|-\varepsilon$.

 $\Longleftarrow$: We must  show that
 for any $\varepsilon>0$, any finite set $F\subseteq A$, and any  positive elements $a, y\in A$ with $a\neq 0$,
  there exist a
${\rm C^*}$-subalgebra $D$ of $A$ with  $D\in \Omega$,  and a  positive element
$g\in D$ with $\|g\|\leq 1$, such that

$(1)$ $\|xg-gx\|<\ep$ for all $x\in  F$,

$(2)$ $gx\in_\ep B, xg\in_\ep D$ for
all $x\in F$,

 $(3)$ $(y^{2}-ypy-\varepsilon)_{+}\precsim _{A} a$, and

 $(4)$ $\|gag\|\geq \|a\|-\varepsilon$.

  For given  $\varepsilon>0$,
 and positive elements  $y, a\in A_+$ with $a\neq 0,$  and as, we may assume, $\|y\|\leq 1$,  by the known conditions,  there exist a
${\rm C^*}$-subalgebra $D$ of $A$  with $D\in \Omega$,
  and a  positive element  $g\in D$ with $\|g\|\leq 1$,
such that

$(1)$ $\|xg-gx\|<\ep$ for all $x\in  F$,

$(2)$ $gx\in_\ep D, xg\in_\ep D$ for
all $x\in  F$,

 $(3)'$ $(1-g-\varepsilon)_+\precsim a$, and

 $(4)$ $\|gag\|\geq \|a\|-\varepsilon$.

By  Theorem 2.2 and $(3)'$, one has

$(3)$ $(y^2-ygy-\varepsilon)_+\precsim (1-g-\varepsilon)_+\precsim a$.
\end{proof}

 \begin{theorem}\label{thm:2.13}
 Let $\Omega$ be a class of ${\rm C^*}$-algebras. Let $A$ be a simple unital ${\rm C^*}$-algebra. Then $A\in {\rm gTA}\Omega$ (see Definition \ref{def:1.2}) if, and only if, for  any $\varepsilon>0$, every finite set $F\subseteq A$, and for every nonzero  positive element $a\in A$, there exist a ${\rm C^*}$-subalgebra $B$ of $A$ with $B\in \Omega$,   and a positive element $g\in B$ with $\|g\|\leq 1$, such that the following conditions hold:

$(1)$ $\|xg-gx\|<\ep$ for all $x\in  F$,

$(2)$ $gx\in_\ep B, xg\in_\ep B$ for
all $x\in  F$,

 $(3)$ $1-g\precsim  a$, and

$(4)$ $\|gag\|\geq \|a\|-\varepsilon$.

 \end{theorem}
 \begin{proof}
 $\Longleftarrow$: This follows from  Theorem \ref{thm:2.12}, and the Cuntz subequivalance  $(1-g-\varepsilon)_+\precsim 1-g$.

 $\Longrightarrow$: We must  show that
 for any $\varepsilon>0$, every finite set $F\subseteq A$ (we assume that $\|x\|\leq 1$ for all $x\in F$), and any positive element $a\in A$ with $a\neq 0$,
  there exist a
${\rm C^*}$-subalgebra $B$ of $A$ with $B\in \Omega$,  and a  positive element  $g\in B$
with $\|g\|\leq 1$,
such that

$(1)$ $\|xg-gx\|<\ep$ for all $x\in  F$,

$(2)$ $gx\in_\ep B, xg\in_\ep B$ for
all $x\in F$,

 $(3)$ $1-g\precsim a$, and

 $(4)$ $\|gag\|\geq \|a\|-\varepsilon$.

  For any   $\delta>0$,
  any finite set $F\subseteq A$, and any  positive  element  $a\in A$ with  $a\neq 0,$  since
   $A\in {\rm gTA}\Omega$, by Theorem \ref{thm:2.12},  there exist a
${\rm C^*}$-subalgebra $B$ of $A$  with $B\in \Omega$,  and a  positive element  $g'\in B$ with $\|g'\|\leq 1$ such that

$(1)'$ $\|xg'-g'x\|<\delta$ for all $x\in  F$,

$(2)'$ $g'x\in_\delta D, xg'\in_\delta D$ for
all $x\in  F$,

 $(3)'$ $(1-g'-\delta)_+\precsim a$, and

 $(4)'$ $\|g'ag'\|\geq \|a\|-\delta$.

 Denote by  $f:[0,1]\rightarrow[0,1]$  the continuous function defined by
\begin{eqnarray*}
	f(\lambda)=
	\begin{cases}
		(1-\delta)^{-1}\lambda,& 0\leqslant\lambda\leqslant 1-\delta, \\
		1,& 1-\delta<\lambda\leqslant1.
	\end{cases}
\end{eqnarray*}
We take $g=f(g')$. Then we have

$(5)'$ $1-g=(1-g'-\delta)_+$, and

$(6)'$ $\|g-g'\|<\delta$.

By $(1)'$, $(2)'$, $(3)'$, $(4)'$, $(5)'$, and  $(6)'$, we have

$(1)$ $\|xg-gx\|\leq \|xg-xg'\|+\|g'x-xg'\|+\|xg'-xg\|<3\delta$ for all $x\in  F$,

$(2)$ $gx\in_{2\delta} B, xg\in_{2\delta} B$ for
all $x\in  F$,

 $(3)$ $1-g=(1-g'-\delta)_+\precsim a$, and

 $(4)$ $\|gag\|\geq \|g'ag'\|-2\delta\geq  \|a\|-3\delta$.

\end{proof}

\begin{theorem}\label{thm:2.14}
 Let $\Omega$ be a class of  unital ${\rm C^*}$-algebras.  Let $A$ be a simple unital ${\rm C^*}$-algebra. Then $A\in {\rm gTA}\Omega$  (see Definition \ref{def:1.2}) if, and only if, $A\in {\rm WTA}\Omega$ (see Definition \ref{def:1.1}).
 \end{theorem}

 \begin{proof}
 $\Longrightarrow$:
  We must show that  for any
 $\varepsilon>0$, any finite
subset $F\subseteq A$, and any  non-zero element $a\geq 0$, there
exist a  projection $p\in A$, an  element $g\in A$ with $0\leq g\leq 1$,
  and a unital ${\rm C^*}$-subalgebra $B$ of $A$ with
$g\in B,~ 1_B=p$, and $B\in \Omega$, such that

$(1)$  $(p-g)x\in_{\varepsilon} B, ~ x(p-g)\in_{\varepsilon} B$ for all $x\in  F$,

$(2)$ $\|(p-g)x-x(p-g)\|<\varepsilon$ for all $x\in F$,

$(3)$ $1-(p-g)\precsim a$, and

$(4)$ $\|(p-g)a(p-g)\|\geq \|a\|-\varepsilon$.

 For  any  $\varepsilon>0$,
 any  finite set $F\subseteq A$, and  any pair of positive elements $a, 1\in A_{+}$ with $a\neq 0$, since $A\in {\rm gTA\Omega}$, by Theorem \ref{thm:2.12}, there exist a
${\rm C^*}$-subalgebra $B$ of $A$ with $B\in \Omega$,
  and a  positive element  $g'\in B$ with $\|g'\|\leq 1$, such that

$(1)'$ $\|xg'-g'x\|<\ep$ for all $x\in  F$,

$(2)'$ $g'x\in_\ep B, xg'\in_\ep B$ for
all $x\in  F$,

 $(3)'$ $(1-g'-\varepsilon)_+\precsim a$, and

$(4)'$ $\|g'ag'\|\geq \|a\|-\varepsilon$.

We take $1_B-g=g'$. By $(1)'$, $(2)'$,  $(3)'$ and $(4)'$, we have

$(1)$  $(1_B-g)x\in_{\varepsilon} B,  x(1_B-g)\in_{\varepsilon} B$ for all $x\in  F$,

$(2)$ $\|(1_B-g)x-x(1_B-g)\|<\varepsilon$ for all $x\in F$,

$(3)$ $1_A-(1_B-g)\precsim a$, and

$(4)$ $\|(1_B-g)a(1_B-g)\|\geq \|a\|-\varepsilon$.

 $\Longleftarrow$:  This follows form Theorem \ref{thm:2.13}.
\end{proof}

The following  theorem is  Lemma 1.7 of \cite{AJN}.
\begin{theorem}{\rm (\cite{AJN}.)}\label{thm:2.15}
For any $\varepsilon>0$ there exists $\delta>0$ such that the following statement holds. Let $A$ be a
${\rm C^*}$-algebra,  $B\subseteq A$ be a ${\rm C^*}$-subalgebra,
 $n$ be  a non-zero integer, $\varphi_0:{\rm M}_n\to A$ be an order zero  completely positive contraction, and  $x\in B$ such that

$(1)$ $0\leq x\leq 1$,

$(2)$ with $(e_{j,k}), j, k=1, 2, \cdots,  n$  the standard system of matrix units for ${\rm M}_n$, we have
$\|\varphi_0(e_{j,k})x-x\varphi_0(e_{j,~k})\|<\varepsilon$, for  $j,~k=1,~2,~\cdots,~ n$, and

$(3)$ $\varphi_0(e_{j,~k})x\in_{\varepsilon}B$.

Then there is an order zero  completely positive contraction  map $\varphi: {\rm M}_n\to B$ such that for all
$z\in {\rm M}_n$, with $\|z\|\leq 1$, one has $\|\varphi_0(z)x-\varphi(z)\|<\varepsilon$.
\end{theorem}

\section{The main results}

\begin{theorem}\label{thm:3.1}Let $\Omega$ be a class of
${\rm C^*}$-algebras which  are tracially $\mathcal{Z}$-absorbing (in the sense of Definition \ref {def:2.6}).  Then  $A$  is  tracially $\mathcal{Z}$-absorbing  for  any non-elementary  simple   ${\rm C^*}$-algebra $A\in{\rm  gTA}\Omega.$
\end{theorem}
\begin{proof}
We must show that
for any finite set $F=\{a_1, a_2, \cdots,  $ $a_k\}\subseteq A,$  any $\varepsilon>0,$  any pari of positive elements $a, b \in A$ with $b\neq 0$,
and any  integer $n\in {{{\mathbb{N}}}},$  there is an order zero  completely positive contraction map $\psi:{\rm M}_n\to A$ such that
the following conditions hold:

$(1)$ $(a^2-a\psi(1)a-\varepsilon)_+\precsim b$, and

$(2)$ for any  element $z\in {\rm M}_n$  of norm one and any $y\in F,$ we have
$\|\psi(z)y-y\psi(z)\|<\varepsilon.$

Since $A$ is non-elementary, and is simple, then by Lemma 2.3 of \cite{P3}, there exist positive  elements $b', b''\in A$  of norm one such that $b'b''=0$,
and $b'+b''\precsim b$.  Also, there exist positive  elements  ${b_1}', {b_2}'\in A$  of norm one such that ${b_1'}{b_2'}=0$, $b_1'\sim b_2'$,
and ${b_1'}+{b_2}'\precsim b'$.

 Given $\varepsilon>0$, with $f(t)=t^{1/2}\in C([0,1])$, there exists $\varepsilon'>0$ satisfying  Lemma 2.5.11 (1) of \cite{L2}. Given such $\varepsilon'>0$,  for $G=F\cup\{b'',  {b''}^{1/2}\}$,  since $A\in{\rm gTA}\Omega$,  there
exist  a  tracially $\mathcal{Z}$-absorbing ${\rm C}^*$-subalgebra $B$ of $A$
and a positive element $g\in B$ with $\|g\|\leq 1$ such that

$(1)'$ $gx\in_{\varepsilon'} B$, $xg\in_{\varepsilon'} B$ for $x\in G$,

$(2)'$ $\|gx-xg\|<\varepsilon'$ for $x\in G$,

$(3)'$ $(a^2-aga-\varepsilon)_+\precsim {b_1}'\sim{b_2}'$, and

$(4)'$ $\|gb''g\|\geq 1-\varepsilon'$.

 By  $(2)'$, with sufficiently small $\varepsilon'$, by  Lemma 2.5.11 (1) of \cite{L2}, we have

$(5)'$ $\|g^{1/2}x-xg^{1/2}\|<\varepsilon,$  for   $x\in G$, and

$(6)'$ $\|(1_{\widetilde{A}}-g)^{1/2}x-x(1_{\widetilde{A}}-g)^{1/2}\|<\varepsilon,$ for   $x\in G$, where $1_{\widetilde{A}}$ is the unit of  ${\widetilde{A}}$ and ${\widetilde{A}}$ is the unitization of $A$.

By $(1)'$, with sufficiently small $\varepsilon'$, together with $(5)'$, there exist  elements $ a_1', a_2', $ $~\cdots, ~a_k'\in B$ and a positive element $b'''\in B$ such that

$(7)'$ $\|g^{1/2}a_ig^{1/2}-a_i'\|<\varepsilon ~~~~ {\rm for}~  1\leq i\leq k$, and

$(8)'$  $\|g^{1/2}b''g^{1/2}-b'''\|<\varepsilon.$

By $(1)'$, together with $(5)'$ and $(6)'$,  for any $1\leq i\leq k$,  one has

$\|a_i-a_i'-(1_{\widetilde{A}}-g)^{1/2}a_i(1_{\widetilde{A}}-g)^{1/2}\|$

$\leq\|a_i-ga_i-(1_{\widetilde{A}}-g)a_i\|+\|ga_i-g^{1/2}a_ig^{1/2}\|$

 $+\|(1_{\widetilde{A}}-g)a_i-(1_{\widetilde{A}}-g)^{1/2}a_i(1_{\widetilde{A}}-g)^{1/2}\|$
 $+\|g^{1/2}a_ig^{1/2}-a_i'\|$

  $<\varepsilon+\varepsilon+\varepsilon=3\varepsilon $.

By $(8)'$ and  by $(1)$ of  Theorem \ref{thm:2.1}, one has

$(9)'$ ~~$(b'''-\varepsilon)_+\precsim g^{1/2}b''g^{1/2}.$

By $(4)'$, if $\varepsilon'\leq \varepsilon$, then
$$\|g^{1/2}b''g^{1/2}\|\geq\|gb''g\|\geq 1-\varepsilon.$$
Hence, by  $(8)'$,
$$1-\varepsilon\leq\|g^{1/2}b''g^{1/2}\|\leq \|b'''\|+\|g^{1/2}b''g^{1/2}-b'''\|$$
$$\leq \|b'''\|+\varepsilon.$$ Therefore,
 $$\|(b'''-\varepsilon)_+\|\geq\|b'''\|-\varepsilon\geq 1-3\varepsilon.$$
 So, if $\varepsilon<1/3$, then  $\|(b'''-\varepsilon)_+\|>0.$

 Since $B\in \Omega,$  with  $H=\{a_1', a_2', \cdots,  a_k',  g^{1/2},  ga_i'\}\subseteq B,$
  $\varepsilon>0$, $(b'''-\varepsilon)_+$ and $g$ with $(b'''-\varepsilon)_+>0$, and any integer  $n$, there  is  an order zero   completely positive contraction map $\psi_0:{\rm M}_n\to B$ such that

$(1)''$ $(g-g^{1/2}\psi_0(1)g^{1/2}-\varepsilon)_+\precsim (b'''-\varepsilon)_+,$ and

$(2)''$ for any element $z\in {\rm M}_n$ of norm one,  and any $x\in H,$ we have
 $\|\psi_0(z)x-x\psi_0(z)\|<\varepsilon.$

By Theorem \ref {thm:2.15},
applied with both the $A$ and $B$ of 2.13 equal to the present $B$,
 there exists  a completely positive contractive order zero map $\psi:{\rm M}_n\to B$ such that
 $\|\psi(z)-\psi_0(z)g\|<\varepsilon.$

 Since  $\|\psi(1)-\psi_0(1)g\|<\varepsilon$, by $(2)''$, one has

 $(3)''$ $\|\psi(1)-g^{1/2}\psi_0(1)g^{1/2}\|<2\varepsilon$.

 Therefore,

$(4)''$ $\|a^2-a\psi(1)a-(a^2-aga)-(aga-ag^{1/2}\psi_0(1)g^{1/2}a)\|<2\varepsilon$.

Then, by $(4)''$,

$(a^2-a\psi(1)a-4\varepsilon)_+$

$\precsim (a^2-aga-\varepsilon)_+\oplus(aga-ag^{1/2}\psi_0(1)g^{1/2}a-\varepsilon)_+$

 $\precsim {b_1}'\oplus (aga-ag^{1/2}\psi_0(1)g^{1/2}a-\varepsilon)_+$

 $\sim {b_1}'\oplus(g-g^{1/2}\psi_0(1)g^{1/2}-\varepsilon)_+, ~~\rm {(by~ Theorem~ 2.2)}$

  $\precsim {b_1}'\oplus(b'''-\varepsilon)_+\precsim b'\oplus g^{1/2}b''g^{1/2}\precsim b'+b''\precsim b$.

This is $(1)$ above, with $4\varepsilon$ in place of  $\varepsilon$.

For any element $z\in {\rm M}_n$ of norm one, any $a_i'\in F$,
we have

$\|\psi(z)a_i'-a_i'\psi(z)\|$

$\leq\|\psi(z)a_i'-\psi_0(z)ga_i'\|
+\|\psi_0(z)ga_i'-ga_i'\psi_0(z)\|$

 $+\|ga_i'\psi_0(z)-gg^{1/2}a_ig^{1/2}\psi_0(z)\|$

 $+\|gg^{1/2}a_ig^{1/2}\psi_0(z)-g^{1/2}a_ig^{1/2}g\psi_0(z)\|$

$+\|g^{1/2}a_ig^{1/2}g\psi_0(z)-a_i'g\psi_0(z)\|$

 $+\|a_i'g\psi_0(z)-a_i'\psi(z)\|<\varepsilon+\varepsilon+\varepsilon+\varepsilon+
 \varepsilon+\varepsilon=6\varepsilon$.

We also have

$\|\psi(z)(1_{\widetilde{A}}-g)^{1/2}a_i(1_{\widetilde{A}}-g)^{1/2}-(1_{\widetilde{A}}-g)^{1/2}a_i(1_{\widetilde{A}}-g)^{1/2}\psi(z)\|$

$\leq\|\psi(z)(1_{\widetilde{A}}-g)^{1/2}a_i((1_{\widetilde{A}}-g)^{1/2}$
$-\psi_0(z)g(1_{\widetilde{A}}-g)^{1/2}a_i(1_{\widetilde{A}}-g)^{1/2}\|$

$+\|\psi_0(z)g(1_{\widetilde{A}}-g)^{1/2}a_i(1_{\widetilde{A}}-g)^{1/2}-\psi_0(z)g(1_{\widetilde{A}}-g)a_i\|$

$+\|\psi_0(z)(1_{\widetilde{A}}-g)ga_i-(1_{\widetilde{A}}-g)\psi_0(z)ga_i\|$

 $+\|(1_{\widetilde{A}}-g)\psi_0(z)ga_i-(1_{\widetilde{A}}-g)\psi_0(z)a_ig\|$

$+\|(1_{\widetilde{A}}-g)\psi_0(z)a_ig-(1_{\widetilde{A}}-g)\psi_0(z)g^{1/2}a_ig^{1/2}\|$

$+\|(1_{\widetilde{A}}-g)\psi_0(z)g^{1/2}a_ig^{1/2}-(1_{\widetilde{A}}-g)\psi_0(z)a_i'\|$

 $+\|(1_{\widetilde{A}}-g)\psi_0(z)a_i'-(1_{\widetilde{A}}-g)a_i'\psi_0(z)\|$

 $+\|(1_{\widetilde{A}}-g)a_i'\psi_0(z)-(1_{\widetilde{A}}-g)g^{1/2}a_ig^{1/2}\psi_0(z)\|$

 $+\|(1_{\widetilde{A}}-g)g^{1/2}a_ig^{1/2}\psi_0(z)-(1_{\widetilde{A}}-g)a_ig\psi_0(z)\|$

 $+\|(1_{\widetilde{A}}-g)a_ig\psi_0(z)-(1_{\widetilde{A}}-g)a_i\psi(z)\|$

 $+\|(1_{\widetilde{A}}-g)a_i\psi(z)-(1_{\widetilde{A}}-g)^{1/2}a_i(1_{\widetilde{A}}-g)^{1/2}\psi(z)\|$

 $<\varepsilon+\varepsilon+\varepsilon+\varepsilon+\varepsilon+\varepsilon+\varepsilon+\varepsilon+\varepsilon+
 \varepsilon+\varepsilon=11\varepsilon$.

Therefore, for any $a_i\in F$ ($1\leq i\leq k$),  we have

$\|\psi(z)a_i-a_i\psi(z)\|$

$\leq\|\psi(z)a_i-\psi(z)(a_i'+(1_{\widetilde{A}}-g)^{1/2}a_i(1_{\widetilde{A}}-g)^{1/2})\|$

$+\|\psi(z)(a_i'+(1_{\widetilde{A}}-g)^{1/2}a_i(1_{\widetilde{A}}-g)^{1/2})$

$-(a_i'+(1_{\widetilde{A}}-g)^{1/2}a_i(1_{\widetilde{A}}-g)^{1/2})\psi(z)\|$

$+\|(a_i'+(1_{\widetilde{A}}-g)^{1/2}a_i(1_{\widetilde{A}}-g)^{1/2})\psi(z)-a_i\psi(z)\|$

 $\leq 3\varepsilon+3\varepsilon+\|\psi(z)a_i'-a_i'\psi(z)\|$

 $+ \|\psi(z)((1_{\widetilde{A}}-g)^{1/2}a_i((1_{\widetilde{A}}-g)^{1/2}$
$-((1_{\widetilde{A}}-g)^{1/2}a_i((1_{\widetilde{A}}-g)^{1/2}\psi(z)\|$

$\leq 6\varepsilon+6\varepsilon+11\varepsilon=23\varepsilon$.

This is $(2)$ above, with $23\varepsilon$ in place of  $\varepsilon$.
\end{proof}
The following corollary was obtained  by Elliott, Fan, and Fang in \cite{EFF}.
\begin{corollary}\label{cor:3.2}Let $\Omega$ be a class of   unital simple
${\rm C^*}$-algebras which  are tracially $\mathcal{Z}$-absorbing (see Definition \ref {def:2.5}).  Then  $A$  is  tracially $\mathcal{Z}$-absorbing  for  any non-elementary simple unital  ${\rm C^*}$-algebra $A\in{\rm  WTA}\Omega.$\end{corollary}
 \begin{proof}By Theorem \ref{thm:3.1}  and Theorem \ref{thm:2.14}.
 \end{proof}

\begin{theorem}\label{thm:3.3}
    	Let $\Omega$ be a class of nuclear $\rm C^{*}$-algebras which have the second  type of  tracial nuclear dimension at
    	most $n$ (in the sense of Definition \ref{def:2.8}). Then ${\rm T^{2}dim_{nuc}} A\leq n$ for any simple $\rm C^{*}$-algebra
    	$A\in {\rm gTA}\Omega$.
    \end{theorem}

    \begin{proof}
    	We must show that for any finite positive subset $\mathcal{F}=\{a_1, a_2,$ $\cdots, a_k\}\subset A$, for any $\varepsilon>0$,
    	and for any $b\in A_+$ with $b\neq 0$, there exist a finite-dimensional $\rm C^{*}$-algebra
    	$F=F_{0}\oplus\cdots\oplus F_{n}$ and  completely positive maps $\psi:A\rightarrow F$, $\varphi:F\rightarrow A$ such that
    	
    	$(1)$ for any  $x\in F$, there exists $\overline{x}\in A_{+}$ such that $\overline{x}\precsim b$ and $\|x-\overline{x}-\varphi\psi(x)\|<\varepsilon$,
    	
    	$(2)$ $\|\psi\|\leq1$, and
    	
    	$(3)$ $\varphi|_{F_{i}}$ is an order zero   completely positive contraction  map for $i=1, \cdots, n$.
    	
    	By Lemma 2.3 of [35], there exist positive elements $b_{1}, b_{2}\in A$ of norm one such that
    	$b_{1}b_{2}=0, b_{1}\sim b_{2}$ and $b_{1}+b_{2}\precsim_{A}b$.
    	
    	Given $\varepsilon'>0$, for $H=\mathcal{F}\cup\{b_2\}$, for $a=a_1+a_2+\cdots+a_k$ and $b_2$, since $A\in {\rm gTA}\Omega$, there exist
    	a $\rm C^{*}$-subalgebra $B$ of $A$ with $B\in\Omega$ and a positive element $g\in B, \|g\|\leq1$, such that
    	
    	$(1)'$ $gx\in_{\varepsilon'}B,  xg\in_{\varepsilon'}B$ for any $x\in H$,
    	
    	$(2)'$ $\|gx-xg\|<\varepsilon'$ for any $x\in H$,
    	
    	$(3)'$ $(a^{2}-aga-\varepsilon)_{+}\precsim b_{1}\sim b_{2}$, and
    	
    	$(4)'$ $\|gb_2g\|\geq 1-\varepsilon'$.
    	
    	By $(2)'$ and Lemma 2.5.11 of [31], with sufficiently small $\varepsilon'$, we can get
    	
    	$(5)'$ $\|g^{\frac{1}{2}}x-xg^{\frac{1}{2}}\|<\varepsilon$ for any $x\in H$, and
    	
    	$(6)'$ $\|(1_{\tilde{A}}-g)^{\frac{1}{2}}x-x(1_{\tilde{A}}-g)^{\frac{1}{2}}\|<\varepsilon$ for any $x\in H$, where $1_{\widetilde{A}}$ is the unital of  ${\widetilde{A}}$ and ${\widetilde{A}}$ is the unitization of $A$.

    	By $(1)'$, together with $(5)'$, with sufficiently small $\varepsilon'$, there exist elements
    $a'_{1}, \cdots, a'_{n}\in B$
    	and a positive element $b_{2}'\in B$ such that $\|g^{\frac{1}{2}}a_{i}g^{\frac{1}{2}}-a'_{i}\|<\varepsilon$ for
    	$1\leq i\leq k$, and $\|g^{\frac{1}{2}}b_{2}g^{\frac{1}{2}}-b_{2}'\|<\varepsilon$.
    	
    	Therefore,
    	
    	 $\|a_{i}-a'_{i}-(1_{\tilde{A}}-g)^{\frac{1}{2}}a_{i}(1_{\tilde{A}}-g)^{\frac{1}{2}}\|$
    	
    	 $\leq\|a_{i}-ga_{i}-(1_{\tilde{A}}-g)a_{i}\|+\|ga_{i}-a'_{i}\|+\|(1_{\tilde{A}}-g)a_{i}-(1_{\tilde{A}}-g)^{\frac{1}{2}}a_{i}(1_{\tilde{A}}-g)^{\frac{1}{2}}\|$
    	
    	$<\varepsilon+\varepsilon+\varepsilon=3\varepsilon$ for $1\leq i\leq n$.
    	
    	Since $\|g^{\frac{1}{2}}b_{2}g^{\frac{1}{2}}-b_{2}'\|<\varepsilon$, by (1) of Theorem \ref{thm:2.1}, we have
    	 $(b_{2}'-3\varepsilon)_{+}\precsim(g^{\frac{1}{2}}b_{2}g^{\frac{1}{2}}-2\varepsilon)_{+}$.
    	
    	By $(4)'$, $\|g^{\frac{1}{2}}b_{2}g^{\frac{1}{2}}\|\geq\|gb_{2}g\|\geq1-\varepsilon$.
    	
    	Therefore, we have $\|(b_{2}'-3\varepsilon)_{+}\|\geq\|g^{\frac{1}{2}}b_{2}g^{\frac{1}{2}}\|+4\varepsilon\geq1-5\varepsilon$,
    	 and then, with $0<\varepsilon'<\varepsilon<\frac{1}{5}$, we have  $(b_{2}'-3\varepsilon)_{+}\neq0$.
    	
    	Define a  completely positive contraction  map $\varphi'':A\rightarrow A$ by $\varphi''(a)=(1_{\tilde{A}}-g)^{\frac{1}{2}}a(1_{\tilde{A}}-g)^{\frac{1}{2}}$. Since $B$ is a nuclear
    	$\rm C^{*}$-algebra, by Theorem 2.3.13 of [30], there exist a  completely positive contraction map $\psi'':A\rightarrow B$ such that
    	$\|\psi''(g)-g\|<\varepsilon$ and $\|\psi''(a'_{i})-a'_{i}\|<\varepsilon$ for all $1\leq i\leq n$.
    	
    	Since $B\in\Omega$, so that  ${\rm T^{2}dim_{nuc}} B\leq n$, there exist a finite-dimensional $\rm C^{*}$-algebra $F=F_{0}\oplus\cdots\oplus F_{n}$ and  completely positive maps $\psi':B\rightarrow F$, $\varphi':F\rightarrow B$ such that
    	
    	$(1)''$  for any $ a_{i}'$, there exists $\overline{\overline{a_{i}'}}\in B_{+}$, such that
    	$\overline{\overline{a_{i}'}}\precsim (b_{2}'-3\varepsilon)_{+}$ and $\|a_{i}'-\overline{\overline{a_{i}'}}-\varphi'\psi'(a_{i}')\|<\varepsilon$, for any $1\leq i\leq k$,
    	
    	$(2)''$ $\|\psi'\|\leq1$, and
    	
    	$(3)''$ $\varphi'|_{F_{i}}$ is an order zero  completely positive contraction map for $i=1, \cdots, n$.
    	
    	Define $\varphi:F\rightarrow A$ by $\varphi(a)=\varphi'(a)$ and $\psi:A\rightarrow F$ by
    	$\psi(a)=\psi'\psi''(g^{\frac{1}{2}}ag^{\frac{1}{2}})$ and
    	 $\overline{a_{i}}=(\varphi''(a_{i})-\varepsilon)_{+}+\overline{\overline{a_{i}'}}$ for $1\leq i\leq n$.

    	For $a_{i}\in\mathcal{F}$, $1\leq i\leq n$, we have
    	
    	 $\overline{a_{i}}=(\varphi''(a_{i})-\varepsilon)_{+}+\overline{\overline{a_{i}'}}=
    	 ((1_{\tilde{A}}-g)^{\frac{1}{2}}a_{i}(1_{\tilde{A}}-g)^{\frac{1}{2}}-\varepsilon)_{+}+\overline{\overline{a_{i}'}}$
    	
    	 $\precsim(a_{i}^{\frac{1}{2}}(1_{\tilde{A}}-g)a_{i}^{\frac{1}{2}}-\varepsilon)_{+}\oplus(b_{2}'-3\varepsilon)_{+}$ ~~(by  $(5)$ of Theorem \ref{thm:2.1})
    	
    	$\precsim b_{1}\oplus(g^{\frac{1}{2}}b_{2}g^{\frac{1}{2}}-2\varepsilon)_{+}$
    	
    	$\precsim b_{1}+b_{2}\precsim_{A}b$.
    	
    	We also have
    	
    	 $\|a_{i}-\overline{a_{i}}-\varphi\psi(a_{i})\|=\|a_{i}-(\varphi''(a_{i})-\varepsilon)_{+}-
    	 \overline{\overline{a_{i}'}}-\varphi'\psi'\psi''(g^{\frac{1}{2}}a_{i}g^{\frac{1}{2}})\|$
    	
    	$\leq 2\varepsilon+\|a_{i}-\varphi''(a_{i})-
    	 \overline{\overline{a_{i}'}}-\varphi'\psi'\psi''(g^{\frac{1}{2}}a_{i}g^{\frac{1}{2}})\|$
    	
    	 $\leq2\varepsilon+\|a_{i}-(1_{\tilde{A}}-g)^{\frac{1}{2}}a_{i}(1_{\tilde{A}}-g)^{\frac{1}{2}}-
    	 \overline{\overline{a_{i}'}}-\varphi'\psi'\psi''(g^{\frac{1}{2}}a_{i}g^{\frac{1}{2}})\|$
    	
    	 $\leq2\varepsilon+\|a_{i}-a'_{i}-(1_{\tilde{A}}-g)^{\frac{1}{2}}a_{i}(1_{\tilde{A}}-g)^{\frac{1}{2}}\|+
    	 \|a'_{i}-\overline{\overline{a_{i}'}}-\varphi'\psi'\psi''(g^{\frac{1}{2}}a_{i}g^{\frac{1}{2}})\|$
    	
    	 $\leq2\varepsilon+\|a_{i}-a'_{i}-(1_{\tilde{A}}-g)^{\frac{1}{2}}a_{i}(1_{\tilde{A}}-g)^{\frac{1}{2}}\|+
    	\|a'_{i}-\overline{\overline{a_{i}'}}-\varphi'\psi'(a_{i}')\|$
    	
    	$+\|\varphi'\psi'(a_{i}')-
    	\varphi'\psi'\psi''(a_{i}')\|+\| \varphi'\psi'\psi''(a_{i}')-\varphi'\psi'(\psi''(g^{\frac{1}{2}}a_{i}g^{\frac{1}{2}})\|$
    	
    	 $<\varepsilon+3\varepsilon+2\varepsilon+2\varepsilon+2\varepsilon=10\varepsilon$.
    	
    	Since $\varphi'', \varphi', \psi', \psi''$ are  completely positive contraction  maps, then $\varphi$ and $\psi$ are completely positive maps.
    	
    	By  $(3)''$, $\varphi'|_{F_{i}}$ is an order zero   completely positive contraction  map, for  $i=1, \cdots, n$, and $\varphi(a)=\varphi'(a)$, and hence
    	$\varphi|_{F_{i}}$ is  an order zero   completely positive contraction map for $i=1,~\cdots,~n$.
    	
    	For any $x\in A$, one has $\|\psi(x)\|=\|\psi'(\psi''g^{\frac{1}{2}}xg^{\frac{1}{2}})\|\leq\|\psi'\|\|\psi''\|\|x\|$, then $\|\psi\|\leq\|\psi''\|\|\psi'\|\leq1$.
    	
    	Therefore, one has ${\rm T^{2}dim_{nuc}} A\leq n$.
    \end{proof}

\begin{theorem}\label{thm:3.4}
    	Let $\Omega$ be a class of $\rm C^{*}$-algebras which are $m$-almost divisible. Let $A\in{\rm gTA}\Omega$
    	be a simple  stably finite $\rm C^{*}$-algebra such that for any $n\in \mathbb{N}$ the $\rm C^{*}$-algebra ${\rm {M}}_{n}(A)$
    	belongs to the class ${\rm gTA}\Omega$. Then $A$ is weakly $(2,m)$-almost divisible.
    \end{theorem}

    \begin{proof} For any given $a\in A_{+},$ any $\varepsilon>0$, and any integer
    	$k\in {\mathbb{N}}$,  (we have replaced ${\rm {M}}_{n}(A)$ containing  a given element $a$ initially by $A$), we must show that there is $b\in {\rm M}_{\infty}(A)_{+}$ such that $k\langle b\rangle\leq\langle a\rangle+\langle a\rangle$
    	and $\langle(a-\varepsilon)_{+}\rangle\leq(k+1)(m+1)\langle b\rangle$.
    	
    	For any $\delta_{1}>0$, with $G=\{a, a^{1/2}\}$, since $A\in{\rm gTA}\Omega$, there exist a $\rm C^{*}$-subalgebra $B$ of $A$ with  $B\in\Omega$, and an element
    	$g\in B$ with $\|g\|\leq1$,  such that
    	
        $(1)$ $ga\in_{\delta_{1}} B, ag\in_{\delta_{1}} B$,  $ga^{1/2}\in_{\delta_{1}} B, a^{1/2} g\in_{\delta_{1}} B$, and
    	
    	$(2)$ $\|ag-ga\|<\delta_{1}$,  $\|a^{1/2}g-ga^{1/2}\|<\delta_{1}$.
    	
    	By $(2)$, with sufficiently small $\delta_{1}$, by Lemma 2.5.11(1) of [31], we have
    	
    	$(3)$ $\|g^{\frac{1}{2}}a-ag^{\frac{1}{2}}\|<\varepsilon/3$,
    	
    	$(4)$ $\|(1_{\widetilde{A}}-g)^{\frac{1}{2}}a-a(1_{\widetilde{A}}-g)^{\frac{1}{2}}\|<\varepsilon/3$, and

       $(5)$ $\|(1_{\widetilde{A}}-g)^{\frac{1}{4}}a^{\frac{1}{2}}-a^{\frac{1}{2}}(1_{\widetilde{A}}-g)^{\frac{1}{4}}\|<\varepsilon/3$,
       where $1_{\widetilde{A}}$ is the unital of  ${\widetilde{A}}$ and ${\widetilde{A}}$ is the unitization of $A$.

    	By $(1)$ and $(2)$, with sufficiently small $\delta_{1}$, there exists a positive $a'\in B$ such that
    	
    	$(6)$ $\|g^{\frac{1}{2}}ag^{\frac{1}{2}}-a'\|<\varepsilon/3$.
    	
    	By $(3)$, $(4)$,  and $(5)$
    	
    	$(7)$ $\|a-a'-(1_{\widetilde{A}}-g)^{\frac{1}{2}}a(1_{\widetilde{A}}-g)^{\frac{1}{2}}\|$

    	 $\leq\|a-ga-(1_{\widetilde{A}}-g)a\|+\|ga-a'\|$

    $+\|(1_{\widetilde{A}}-g)a-(1_{\widetilde{A}}-g)^{\frac{1}{2}}a(1_{\widetilde{A}}-g)^{\frac{1}{2}}\|$

    	$<\varepsilon/3+\varepsilon/3+\varepsilon/3=\varepsilon$ .
    	
    	Since $B$ is $m$-almost divisible, and $(a'-3\varepsilon)_{+}\in B$, there exists $b_{1}\in B$ such that
    	$k\langle b_{1}\rangle\leq\langle(a'-3\varepsilon)_{+}\rangle$ and $\langle(a'-4\varepsilon)_{+}\rangle\leq(k+1)(m+1)\langle b_{1}\rangle$.
    	
    	Since $B$ is $m$-almost divisible, and $(a'-2\varepsilon)_{+}\in B$, there exists $b'\in B$ such that
    	$k\langle b'\rangle\leq\langle(a'-2\varepsilon)_{+}\rangle$ and $\langle(a'-3\varepsilon)_{+}\rangle\leq(k+1)(m+1)\langle b'\rangle$.
    	
    	Write $a''=(1_{\widetilde{A}}-g)^{\frac{1}{2}}a(1_{\widetilde{A}}-g)^{\frac{1}{2}}$.
    	
    	 We divide the proof into two cases.
    	
    	\textbf{{Case (1)}} We assume that $(a'-3\varepsilon)_{+}$ is Cuntz eqivalent to a projection.
    	
    	\textbf{{ (1.1)}} We assume that $(a'-4\varepsilon)_{+}$ is Cuntz eqivalent to a projection.
    	
    	\textbf{{ (1.1.1)}} If $\langle(a'-4\varepsilon)_{+}\rangle$ is not equal to
    	$(k+1)(m+1)\langle b_{1}\rangle$, there exists a non-zero element $c\in A_{+}$ such that
    	$\langle(a'-4\varepsilon)_{+}\rangle+\langle c\rangle\leq(k+1)(m+1)\langle b_{1}\rangle$.
    	
    	For any $0<\delta_{2}<\varepsilon$,  with $F_1=\{a'', (1_{\widetilde{A}}-g)^{\frac{1}{4}}a^{\frac{1}{2}},  a^{\frac{1}{2}}(1_{\widetilde{A}}-g)^{\frac{1}{4}},
    	(1_{\widetilde{A}}-g)^{\frac{1}{2}}a^{\frac{1}{2}},
    a^{\frac{1}{2}}(1_{\widetilde{A}}-g)^{\frac{1}{2}},
    	 (1_{\widetilde{A}}-g)^{\frac{1}{4}}a^{\frac{1}{2}}(1_{\widetilde{A}}-g)^{\frac{1}{4}}a^{\frac{1}{2}},
    	a^{\frac{1}{2}}(1_{\widetilde{A}}-g)^{\frac{1}{2}}a^{\frac{1}{2}}\}$, with $c$ and $(1_{\widetilde{A}}-g)^{\frac{1}{4}}a^{\frac{1}{2}}(1_{\widetilde{A}}-g)^{\frac{1}{4}}$,
    	since $A\in{\rm gTA}\Omega$, there exist a $\rm C^{*}$-subalgebra $D$ of $A$ with
    $D\in\Omega$, and a positive  element $g_{1}\in D$ with $\|g_1\|\leq 1$, such that
    	
    	$(1)'$ $g_{1}x\in_{\delta_{2}} D, xg_{1}\in_{\delta_{2}} D$, for any $x\in F_1$,
    	
    	$(2)'$ $\|xg-gx\|<\delta_{2}$, for any $x\in F_1$, and
    	
    	$(3)'$ $(((1_{\widetilde{A}}-g)^{\frac{1}{4}}a^{\frac{1}{2}}(1_{\widetilde{A}}-g)^{\frac{1}{4}})^{2}-
    (1_{\widetilde{A}}-g)^{\frac{1}{4}}a^{\frac{1}{2}}(1_{\widetilde{A}}-g)^{\frac{1}{4}}g_{1}
    (1_{\widetilde{A}}-g)^{\frac{1}{4}}a^{\frac{1}{2}}(1_{\widetilde{A}}-g)^{\frac{1}{4}}-\varepsilon)_{+}\precsim c$.
    	
    	By $(1)'$ and $(2)'$, with sufficiently small $\delta_{2}$, as above, via the analogues of $(4)$, $(5)$
    	for $a'', g_{1}$, there exists a positive element $a'''\in D$ such that:
    	
    	 $(4)'$ $\|(1_{\widetilde{A}}-g)^{\frac{1}{4}}a^{\frac{1}{2}}(1_{\widetilde{A}}-g)^{\frac{1}{4}}g_{1}^{\frac{1}{2}}-
    	 g_{1}^{\frac{1}{2}}(1_{\widetilde{A}}-g)^{\frac{1}{4}}a^{\frac{1}{2}}(1_{\widetilde{A}}-g)^{\frac{1}{4}}\|<\varepsilon$,
    	
    	 $(5)'$ $\|(1_{\widetilde{A}}-g)^{\frac{1}{4}}a^{\frac{1}{2}}(1_{\widetilde{A}}-g)^{\frac{1}{4}}(1_{\widetilde{A}}-g_{1})^{\frac{1}{2}}-
    	 (1_{\widetilde{A}}-g_{1})^{\frac{1}{2}}(1_{\widetilde{A}}-g)^{\frac{1}{4}}a^{\frac{1}{2}}(1_{\widetilde{A}}-g)^{\frac{1}{4}}\|<\varepsilon$,
    	
    	 $(6)'$ $\|g_{1}^{\frac{1}{2}}a''g_{1}^{\frac{1}{2}}-a'''\|<\varepsilon/3$, and

    	  $(7)'$ $\|a''-a'''-(1_{\widetilde{A}}-g_{1})^{\frac{1}{2}}a''(1_{\widetilde{A}}-g_{1})^
    {\frac{1}{2}}\|<\varepsilon$.
    	
    	Since $D$ is $m$-almost divisible, and $(a'''-3\varepsilon)_{+}\in D$, there exists $b_{2}\in D$ such that
    	$k\langle b_{2}\rangle\leq\langle(a'''-3\varepsilon)_{+}\rangle$ and $\langle(a'''-4\varepsilon)_{+}\rangle\leq(k+1)(m+1)\langle b_{2}\rangle$.
    	
    	Since $a'\leq a'+a''$,  by Theorem \ref{thm:2.1} (4), one has $\langle(a'-\varepsilon)_{+}\rangle\leq\langle(a'+a''-\varepsilon)_{+}\rangle$,  Also, by $(7)$,
    	one has  $\langle(a'+a''-\varepsilon)_{+}\rangle\leq\langle a\rangle$,  and therefore $\langle(a'-\varepsilon)_{+}\rangle\leq\langle a\rangle$,  and similarly $\langle(a'''-3\varepsilon)_{+}\rangle\leq\langle a\rangle$.
    	
    	 Therefore,  we have

    	$k\langle b_{1}\oplus b_{2}\rangle=k\langle b_{1}\rangle+k\langle b_{2}\rangle$ \\
        $\leq\langle(a'-3\varepsilon)_{+}\rangle+\langle(a'''-3\varepsilon)_{+}\rangle$ \\
        $\leq\langle(a'-2\varepsilon)_{+}\rangle+\langle(a'''-3\varepsilon)_{+}\rangle$ \\
        $\leq\langle a\rangle+\langle a\rangle$.

        By $(7)$ and $(7)'$, one has

        $(8)'$ $\|a-a'-a'''-(1_{\widetilde{A}}-g_{1})^{\frac{1}{2}}a''(1_{\widetilde{A}}-g_{1})^{\frac{1}{2}}\|<2\varepsilon$.

        Then, by $(4)'$ and $(8)'$, one has

                $(9)'$ $\|a-a'-a'''-(1_{\widetilde{A}}-g)^{\frac{1}{4}}a^{\frac{1}{2}}(1_{\widetilde{A}}-g)^{\frac{1}{4}}(1_{\widetilde{A}}-g_{1})
        (1_{\widetilde{A}}-g)^{\frac{1}{4}}a^{\frac{1}{2}}(1_{\widetilde{A}}-g)^{\frac{1}{4}}\|<8\varepsilon$.

         Then,  by $(9)'$ and $(3)'$,  we can get

         $\langle(a-20\varepsilon)_{+}\rangle$

         $\leq\langle(a'-4\varepsilon)_{+}\rangle+\langle(a'''-4\varepsilon)_{+}\rangle$

         $+\langle(((1_{\widetilde{A}}-g)^{\frac{1}{4}}a^{\frac{1}{2}}(1_{\widetilde{A}}-g)^{\frac{1}{4}})^{2}$

         $-(1_{\widetilde{A}}-g)^{\frac{1}{4}}a^{\frac{1}{2}}(1_{\widetilde{A}}-g)^{\frac{1}{4}}g_{1}
         (1_{\widetilde{A}}-g)^{\frac{1}{4}}a^{\frac{1}{2}}(1_{\widetilde{A}}-g)^{\frac{1}{4}}-\varepsilon)_{+}\rangle$

         $\leq\langle(a'-4\varepsilon)_{+}\rangle+\langle(a'''-4\varepsilon)_{+}\rangle+\langle c\rangle$

         $\leq(k+1)(m+1)\langle b_{1}\rangle+(k+1)(m+1)\langle b_{2}\rangle=(k+1)(m+1)\langle b_{1}\oplus b_{2}\rangle$.

         These are the desired inequalities, with $b_{1}\oplus b_{2}$ in place of $b$ and $20\varepsilon$ in place of $\varepsilon$.

        \textbf{ { (1.1.2)}} If $\langle(a'-4\varepsilon)_{+}\rangle\sim(k+1)(m+1)\langle b_{1}\rangle$, so that
         $(k+1)(m+1)\langle b_{1}\rangle\leq\langle(a'-3\varepsilon)_{+}\rangle$, then as $m+1\geq2$, we have
         $k\langle b_{1}\oplus b_{1}\rangle\leq\langle(a'-3\varepsilon)_{+}\rangle$ and
         $\langle(a'-4\varepsilon)_{+}\rangle+\langle b_{1}\rangle\leq(k+1)(m+1)\langle b_{1}\oplus b_{1}\rangle$.

         As in the part \textbf{ (1.1.1)}, as $A\in{\rm gTA}\Omega$, with $b_{1}$ in place of $c$ in the part \textbf{ (1.1.1)},  there exists a $\rm C^{*}$-subalgebra $D_{1}$ of $A$ with $D_{1}\in\Omega$,
          and a positive element $g_{2}\in D_{1}$ with $\|g_{2}\|\leq 1$,
         there exists a positive element $a^{(4)}\in D_{1}$ such that

    	 $(10)'$ $\|a''-a^{(4)}-(1_{\widetilde{A}}-g_{2})^{\frac{1}{2}}a''(1_{\widetilde{A}}-g_{2})^{\frac{1}{2}}\|<\varepsilon$,

    	$(11)'$$\langle(a^{(4)}-3\varepsilon)_{+}\rangle\leq\langle a \rangle$,  and

    $(12)'$$\langle(((1_{\widetilde{A}}-g)^{\frac{1}{4}}a^{\frac{1}{2}}(1_{\widetilde{A}}-g)^{\frac{1}{4}})^{2}-
    (1_{\widetilde{A}}-g)^{\frac{1}{4}}a^{\frac{1}{2}}(1_{\widetilde{A}}-g)^{\frac{1}{4}}g_{2}(1_{\widetilde{A}}-g)^{\frac{1}{4}}a^{\frac{1}{2}}(1_{\widetilde{A}}-g)^{\frac{1}{4}}-\varepsilon)_{+}\rangle\precsim \langle b_{1}\rangle$.
    	
    	Since $D_{1}$ is $m$-almost divisible and $(a^{(4)}-3\varepsilon)_{+}\in D_{1}$, there exists $b_{3}\in (D_{1})_{+}$ such that
    	$k\langle b_{3}\rangle\leq\langle(a^{(4)}-3\varepsilon)_{+}\rangle$ and $\langle(a^{(4)}-4\varepsilon)_{+}\rangle\leq(k+1)(m+1)\langle b_{3}\rangle$.
    	
    	Therefore, we have
    	
    	$k\langle b_{1}\oplus b_{1}\oplus b_{3}\rangle=k\langle b_{1}\oplus b_{1}\rangle+k\langle b_{3}\rangle$

    	 $\leq\langle(a'-3\varepsilon)_{+}\rangle+\langle(a^{(4)}-3\varepsilon)_{+}\rangle$

    	$\leq\langle a\rangle+\langle a\rangle$.

    By $(7)'$ and $(10)'$, one has

    $(13)'$ $\|a-a'-a^{(4)}-(1_{\widetilde{A}}-g_{2})^{\frac{1}{2}}a''(1_{\widetilde{A}}-g_{2})^{\frac{1}{2}}\|<2\varepsilon$,

    Then,  By $(7)'$ and $(10)'$,
    	
    	 $(14)'$$\|a-a'-a^{(4)}-(1_{\widetilde{A}}-g)^{\frac{1}{4}}a^{\frac{1}{2}}(1_{\widetilde{A}}-g)^{\frac{1}{4}}(1_{\widetilde{A}}-g_{2})
    	 (1_{\widetilde{A}}-g)^{\frac{1}{4}}a^{\frac{1}{2}}(1_{\widetilde{A}}-g)^{\frac{1}{4}}\|<8\varepsilon$.
    	
    	Then, By $(7)'$ and $(10)'$, we can get
    	
    	$\langle(a-20\varepsilon)_{+}\rangle$

    	 $\leq\langle(a'-4\varepsilon)_{+}\rangle$

    $+\langle(a^{(4)}-4\varepsilon)_{+}\rangle+    	 \langle(((1_{\widetilde{A}}-g)^{\frac{1}{4}}a^{\frac{1}{2}}(1_{\widetilde{A}}-g)^{\frac{1}{4}})^{2}$

    $-(1_{\widetilde{A}}-g)^{\frac{1}{4}}a^{\frac{1}{2}}
    (1_{\widetilde{A}}-g)^{\frac{1}{4}}g_{2}(1_{\widetilde{A}}-g)^{\frac{1}{4}}a^{\frac{1}{2}}(1_{\widetilde{A}}-g)^{\frac{1}{4}}
    )-\varepsilon)_{+}\rangle$

    	 $\leq\langle(a'-4\varepsilon)_{+}\rangle+\langle(a^{(4)}-4\varepsilon)_{+}\rangle+\langle b_{1}\rangle$

    	$\leq(k+1)(m+1)\langle b_{1}\oplus b_{1}\rangle+(k+1)(m+1)\langle b_{3}\rangle=(k+1)(m+1)\langle b_{1}\oplus b_{1}\oplus b_{3}\rangle$.
    	
    	These are the desired inequalities, with $b_{1}\oplus b_{1}\oplus b_{3}$ in place of $b$ and $20\varepsilon$ in place of $\varepsilon$.
    	
    	\textbf{{(1.2)}} We assume that $(a'-4\varepsilon)_{+}$ is not Cuntz equivalent to a projection. By Theorem \ref{thm:2.1} (3), there is     	 a non zero positive element $d$ such that $\langle(a'-5\varepsilon)_{+}\rangle+\langle d\rangle\leq\langle(a'-4\varepsilon)_{+}\rangle$.
    	
    	 As in the part \textbf{(1.1.1)}, as $A\in{\rm gTA}\Omega$, with $d$ in place of $c$ in the part\textbf{ (1.1.1)}, there exists a $\rm C^{*}$-subalgebra $D_{2}$ of $A$ with $D_{2}\in\Omega$, and a positive  element $g_{3}\in D_{2}$ with $\|g_{3}\|\leq 1$, and
    	 there exists a positive element $a^{(5)}\in D_{2}$ such that
    	
    	 $(1)''$  $\|a''-a^{(5)}-(1_{\widetilde{A}}-g_{3})^{\frac{1}{2}}a''(1_{\widetilde{A}}-g_{3})^{\frac{1}{2}}\|<\varepsilon$,

    	$(2)''$  $\langle(a^{(5)}-3\varepsilon)_{+}\rangle\leq\langle a \rangle$,
     and

     $(3 )''$ $\langle((((1_{\widetilde{A}}-g)^{\frac{1}{4}}a^{\frac{1}{2}}(1_{\widetilde{A}}-g)^{\frac{1}{4}})^{2}-(1_{\widetilde{A}}-g)^{\frac{1}{4}}a^{\frac{1}{2}}(1_{\widetilde{A}}-g)^{\frac{1}{4}}g_{3}(1_{\widetilde{A}}-g)^{\frac{1}{4}}a^{\frac{1}{2}}(1_{\widetilde{A}}-g)^{\frac{1}{4}})-\varepsilon)_{+}\rangle\precsim \langle d\rangle$.
    	
    	Since $D_{2}$ is $m$-almost divisible, and $(a^{(5)}-3\varepsilon)_{+}\in D_{2}$, there exists $b_{4}\in (D_{2})_{+}$ such that
    	$k\langle b_{4}\rangle\leq\langle(a^{(5)}-3\varepsilon)_{+}\rangle$ and $\langle(a^{(5)}-4\varepsilon)_{+}\rangle\leq(k+1)(m+1)\langle b_{4}\rangle$.
    	
    	Therefore,
    	
    	  $k\langle b_{1}\oplus b_{4}\rangle=k\langle b_{1}\rangle+k\langle b_{4}\rangle$

        	 $\leq\langle(a'-3\varepsilon)_{+}\rangle+\langle(a^{(5)}-3\varepsilon)_{+}\rangle$

    	$\leq\langle a\rangle+\langle a\rangle$.
    	
    	 By $(7)$ and $(1)''$, one has

     $(4 )''$ $\|a-a'-a^{(5)}-(1_{\widetilde{A}}-g_{3})^{\frac{1}{2}}a''(1_{\widetilde{A}}-g_{3})^{\frac{1}{2}}\|<2\varepsilon$,
    	
    By $(3)''$ and $(4)''$, one has
    	
    	 $(5)''$ $\|a-a'-a^{(5)}-(1_{\widetilde{A}}-g)^{\frac{1}{4}}a^{\frac{1}{2}}(1_{\widetilde{A}}-g)^{\frac{1}{4}}(1_{\widetilde{A}}-g_{3})
    	 (1_{\widetilde{A}}-g)^{\frac{1}{4}}a^{\frac{1}{2}}(1_{\widetilde{A}}-g)^{\frac{1}{4}}\|<8\varepsilon$.
    	
    	Then, by $(7)$ and  $(5)''$,  we can get
    	
    	$\langle(a-20\varepsilon)_{+}\rangle$

    	 $\leq\langle(a'-5\varepsilon)_{+}\rangle+\langle(a^{(5)}-4\varepsilon)_{+}\rangle$

    $+\langle(((1_{\widetilde{A}}-g)^{\frac{1}{4}}a^{\frac{1}{2}}(1_{\widetilde{A}}-g)^{\frac{1}{4}})^{2}$

    $-(1_{\widetilde{A}}-g)^{\frac{1}{4}}a^{\frac{1}{2}}(1_{\widetilde{A}}-g)^
    {\frac{1}{4}}g_{3}(1_{\widetilde{A}}-g)^{\frac{1}{4}}a^{\frac{1}{2}}(1_{\widetilde{A}}-g)^{\frac{1}{4}}-\varepsilon)_{+}\rangle$

    	 $\leq\langle(a'-5\varepsilon)_{+}\rangle+\langle(a^{(5)}-4\varepsilon)_{+}\rangle+\langle d\rangle$

    	$\leq(k+1)(m+1)\langle b_{1}\rangle+(k+1)(m+1)\langle b_{4}\rangle=(k+1)(m+1)\langle b_{1}\oplus b_{4}\rangle$.
    	
    	These are the desired inequalities, with $b_{1}\oplus b_{4}$ in place of $b$, and $20\varepsilon$ in place of $\varepsilon$.
    	
    	\textbf{{ Case (2)}} We assume that $(a'-3\varepsilon)_{+}$ is not Cuntz equivalent to a projection.
    By Theorem \ref{thm:2.1} (3),  there is a non-zero positive element $r$ such that $\langle(a'-4\varepsilon)_{+}\rangle+\langle r\rangle\leq\langle(a'-3\varepsilon)_{+}\rangle$.
    	
    	 As in the part \textbf{(1.1.1)}, as $A\in{\rm gTA}\Omega$, with $r$ in place of $c$ in the part\textbf{ (1.1.1)},
    there exists a $\rm C^{*}$-subalgebra $D_{3}$ of $A$ with $D_{3}\in\Omega$,  and an element $g_{4}\in D_{3}$ with $\| g_{4}\|\leq 1$, and
    	 there exists a positive element $a^{(6)}\in D_{3}$,  such that
    	
    	 $(1 )'''$ $\|a''-a^{(6)}-(1_{\widetilde{A}}-g_{4})^{\frac{1}{2}}a''(1_{\widetilde{A}}-g_{4})^{\frac{1}{2}}\|<\varepsilon$,

    	$(2 )'''$  $\langle(a^{(6)}-3\varepsilon)_{+}\rangle\leq\langle a \rangle$, and

    $(3 )'''$ $\langle((((1_{\widetilde{A}}-g)^{\frac{1}{4}}a^{\frac{1}{2}}(1_{\widetilde{A}}-g)^{\frac{1}{4}})^{2}-(1_{\widetilde{A}}-g)^{\frac{1}{4}}a^{\frac{1}{2}}(1_{\widetilde{A}}-g)^{\frac{1}{4}}g_{4}(1_{\widetilde{A}}-g)^{\frac{1}{4}}a^{\frac{1}{2}}(1_{\widetilde{A}}-g)^{\frac{1}{4}})-\varepsilon)_{+}\rangle\precsim \langle r\rangle$.
    	
    	 Since $D_{3}$ is $m$-almost divisible, and $(a^{(6)}-3\varepsilon)_{+}\in D_{3}$, there exists $b_{5}\in (D_{3})_{+}$ such that
    	 $k\langle b_{5}\rangle\leq\langle(a^{(6)}-3\varepsilon)_{+}\rangle$ and $\langle(a^{(6)}-4\varepsilon)_{+}\rangle\leq(k+1)(m+1)\langle b_{5}\rangle$.
    	
    	 Therefore,
    	
    	 $k\langle b'\oplus b_{5}\rangle=k\langle b'\rangle+k\langle b_{5}\rangle$

    	 $\leq\langle(a'-3\varepsilon)_{+}\rangle+\langle(a^{(6)}-3\varepsilon)_{+}\rangle$

    	 $\leq\langle a\rangle+\langle a\rangle$.
    	
    	  By $(7)$ and $(1 )'''$, one has

      $(4 )'''$  $\|a-a'-a^{(6)}-(1_{\widetilde{A}}-g_{4})^{\frac{1}{2}}a''(1_{\widetilde{A}}-g_{4})^{\frac{1}{2}}\|<2\varepsilon$.
    	
    By $(3)'''$ and $(4)'''$, one has
    	
    	 $(5 )'''$ $\|a-a'-a^{(6)}
    -(1_{\widetilde{A}}-g)^{\frac{1}{4}}a^{\frac{1}{2}}(1_{\widetilde{A}}-g)^{\frac{1}{4}}(1_{\widetilde{A}}-g_{4})
    	 (1_{\widetilde{A}}-g)^{\frac{1}{4}}a^{\frac{1}{2}}(1_{\widetilde{A}}-g)^{\frac{1}{4}}\|<8\varepsilon$.
    	
    	 Then, by $(7)$ and $(5 )'''$,  we can get
    	
    	 $\langle(a-20\varepsilon)_{+}\rangle$

    	 $\leq\langle(a'-4\varepsilon)_{+}\rangle+\langle(a^{(6)}-4\varepsilon)_{+}\rangle$

    $+\langle(((1_{\widetilde{A}}-g)^{\frac{1}{4}}a^{\frac{1}{2}}(1_{\widetilde{A}}-g)^{\frac{1}{4}})^{2}$

    $-(1_{\widetilde{A}}-g)^{\frac{1}{4}}a^{\frac{1}{2}}(1_{\widetilde{A}}-g)^{\frac{1}{4}}
    g_{4}(1_{\widetilde{A}}-g)^{\frac{1}{4}}a^{\frac{1}{2}}(1_{\widetilde{A}}-g)^{\frac{1}{4}}-\varepsilon)_{+}\rangle$

    	 $\leq\langle(a'-4\varepsilon)_{+}\rangle+\langle(a^{(6)}-4\varepsilon)_{+}\rangle+\langle r\rangle$

    	 $\leq(k+1)(m+1)\langle b'\rangle+(k+1)(m+1)\langle b_{5}\rangle=(k+1)(m+1)\langle b'\oplus b_{5}\rangle$.
    	
    	 These are the desired inequalities, with $b'\oplus b_{5}$ in place of $b$ and $20\varepsilon$ in place of $\varepsilon$.
          	
    \end{proof}
The proof of the following theorem is the same as Theorem 3.1 in \cite{EFF}.
 \begin{theorem}\label{thm:3.5}
    	Let $\Omega$ be a class of $\rm C^{*}$-algebras which have the property $\rm {SP}$. Then $A$ has the property $\rm {SP}$ for any simple unital $\rm C^{*}$-algebra $A\in{\rm gTA}\Omega$.
    \end{theorem}

    \begin{proof}
    	Let $B$ be a non-zero hereditary $\rm C^{*}$-subalgebra of $A$. We must show that $B$ contains a non-zero projection.
    	Choose a positive element $a$ of $B$, $\|a\|=1$.
    	
    	Let $0<\varepsilon<1$ be two positive numbers. Define
    	\begin{eqnarray*}
    		f_{\varepsilon}(t)=
    		\begin{cases}
    			1,& \text{if}\quad t\geq\varepsilon  \\
    			(2t-\varepsilon)/\varepsilon,& \text{if} \quad\varepsilon/2<t\leq\varepsilon \\
    			0,& \text{if} \quad0\leq t\leq\varepsilon/2.
    		\end{cases}
    	\end{eqnarray*}
        	
    	Let $F=\{a,~a^{1/2}\}$, and any $\varepsilon'>0$. Since $A\in{\rm gTA}\Omega$, there exists a
    	$\rm C^{*}$-subalgebra $D$ of $A$ with $D\in\Omega$, and a positive element $g\in D$, such that the following  conditions hold:
    	
    	$(1)$ $\|xg-gx\|<\ep'$, for any $x\in F$,
    	
    	$(2)$ $gx\in_{\ep'} B, xg\in_{\ep'} B$, for any $x\in F$, and

    	$(3)$ $\|gag\|\geq 1-\varepsilon'$.
    	
    	By $(1)$ and $(2)$, with sufficiently small $\varepsilon'$, there exists an element of norm at most one
    	$a'\in D_{+}$, such that $\|a^{\frac{1}{2}}g^{2}a^{\frac{1}{2}}-a'\|<\delta_{2}$, $\|gag-a'\|<\delta_{2}$.
    	
    	Since $\|gag-a'\|<\delta_{2}$, $\|gag\|\geq1-\varepsilon'$, one has $(a'-\varepsilon)_{+}\neq0$
    	(otherwise,$1-\varepsilon'<\delta_{2}+\varepsilon$), since $D\in\Omega$, $D$ has the property $\rm {SP}$, and so
    	there exists a non-zero projection $p\in \overline{(a'-\varepsilon)_{+}D(a'-\varepsilon)_{+}}$.
    	
    	Since $f_{\varepsilon}(a')(a'-\varepsilon)_{+}=(a'-\varepsilon)_{+}$, we have $f_{\varepsilon}(a')p=p$,
    	 and since  $\|a^{\frac{1}{2}}g^{2}a^{\frac{1}{2}}-a'\|<\delta_{2}$, we get
    	 $\|f_{\varepsilon}(a^{\frac{1}{2}}g^{2}a^{\frac{1}{2}})-f_{\varepsilon}(a')\|<\varepsilon$.
    	
    	Then we have
        $	 \|f_{\varepsilon}(a^{\frac{1}{2}}g^{2}a^{\frac{1}{2}})pf_{\varepsilon}(a^{\frac{1}{2}}g^{2}a^{\frac{1}{2}})-p\|$ \\
        $	 =\|f_{\varepsilon}(a^{\frac{1}{2}}g^{2}a^{\frac{1}{2}})pf_{\varepsilon}(a^{\frac{1}{2}}g^{2}a^{\frac{1}{2}})-
        	f_{\varepsilon}(a')pf_{\varepsilon}(a')<3\varepsilon $.

       With sufficiently small $\varepsilon$, there exists a non-zero projection  $q\in {\rm Her(a)}$. Since
       $a\in B$ and  $B$ is hereditary,  $q\in B$. This show that  $A$ has the property $\rm {SP}$.
    \end{proof}

 \end{document}